\documentclass[a4paper,11pt,reqno]{amsart}
\usepackage{epsfig,url,paralist}
\usepackage{amssymb,ulem}
\usepackage{amsmath}
\usepackage{amsthm, tikz,tikz-cd}
\usepackage{hyperref, cleveref,verbatim}
\usepackage{graphicx}
\usepackage{setspace}
\usepackage[shortlabels]{enumitem}
\setlist{nolistsep}
\usepackage{lscape}
\usepackage{fancyhdr}
\usepackage{parskip}
\usepackage{mathtools}
\numberwithin{equation}{section}

\newtheorem{theorem}{Theorem}
\newtheorem*{theorem*}{Main Theorem}
\newtheorem{lem}[theorem]{Lemma}
\newtheorem{coro}[theorem]{Corollary}
\newtheorem{prop}[theorem]{Proposition}

\theoremstyle{definition}

\newtheorem{remark}[theorem]{Remark}

\numberwithin{theorem}{section}

\theoremstyle{plain}
\newtheorem{thmx}{Theorem}

\newtheorem{prpx}[thmx]{Proposition}

\newcommand{\tr}{\mathrm{tr}}

\newcommand{\slc}{{\rm SL_2}(\mathbb{C})}
\newcommand{\slk}{{\rm SL_2}(K)}
\newcommand{\slki}{{\rm SL_2}(K(i))}

\newcommand{\chk}{\mathrm{char}(K)}
\newcommand{\pslk}{{\rm PSL_2}(K)}
\newcommand{\pslki}{{\rm PSL_2}(K(i))}

\newcommand{\pslc}{{\rm PSL_2}(\mathbb{C})}

\newcommand{\qp}{\mathbb{Q}_p}
\newcommand{\fq}{\mathbb{F}_q}
\newcommand{\fp}{\mathbb{F}_p}

\newcommand{\zz}{\mathbb{Z}}

\newcommand{\oo}{\mathcal{O}_K}

\newcommand{\fix}{{\rm Fix}}
\newcommand{\ax}{{\rm Ax}}

\newcommand{\G}{\mathbb{G}}

\newcommand\setsymbol[1][]{\nonscript\:#1\vert\allowbreak\nonscript\:\mathopen{}}
\providecommand\given{} 
\DeclarePairedDelimiterX{\set}[1]\{\}{\renewcommand\given{\setsymbol[\delimsize]}#1}
\DeclarePairedDelimiterX{\gen}[1]{\langle}{\rangle}{\renewcommand\given{\setsymbol[\delimsize]}#1}

\begin{document}
\title{Lifting subgroups of $\mathrm{PSL}_2$ to $\mathrm{SL}_2$ over local fields}

\author[Andrew]{Naomi Andrew}
\author[Conder]{Matthew J. Conder}
\author[Markowitz]{Ari Markowitz}
\author[Schillewaert]{Jeroen Schillewaert}

\address[Andrew]{Mathematical Institute, Andrew Wiles Building, Radcliffe Observatory Quarter, University of Oxford, Oxford OX2 6GG, United Kingdom}
\address[Conder, Markowitz and Schillewaert]{Department of Mathematics, University of Auckland, 38 Princes Street, 1010 Auckland, New Zealand}

\email{naomi.andrew@maths.ox.ac.uk}
\email{\{matthew.conder,ari.markowitz,j.schillewaert\}@auckland.ac.nz}


\maketitle

\begin{abstract}
Let $K$ be a non-archimedean local field.
We show that discrete subgroups without 2-torsion in $\pslk$ can always be lifted to $\slk$, and provide examples (when $\chk \neq 2$) which cannot be lifted if either of these conditions is removed.
We also briefly discuss lifting representations of groups into $\pslk$ to $\slk$.
\end{abstract}

\section{Introduction}

Let $\pi \colon \Gamma_1 \twoheadrightarrow \Gamma_2$ be a surjective group homomorphism. We say that a subgroup $G$ of $\Gamma_2$ {\it lifts} to $\Gamma_1$ if there exists a homomorphism $\phi \colon G \to \Gamma_1$ such that $\pi \circ \phi$ is the identity on $G$.

Here we consider the case of $\pi \colon \slk \twoheadrightarrow \pslk$ over a field $K$. If ${\rm char}(K)=2$, then $\slk \cong \pslk$, so every subgroup of $\pslk$ lifts to $\slk$. Hence we will assume ${\rm char}(K) \neq 2$. In this case, $-I$ is the unique element of order 2 in $\slk$, so any subgroup of $\pslk$ containing elements of order 2 necessarily cannot lift to $\slk$. 

Let $K$ be a finitely generated subgroup of $\mathrm{SL}(2,\bar{\mathbb{Q}})$, where $\bar{\mathbb{Q}}$ are the algebraic numbers. Even though $K$ might not be discrete in $\mathrm{SL}(2,\mathbb{C})$ it embeds discretely in a finite direct product with factors $\mathrm{SL}(2,k_i)$ where the $k_i$ are local fields determined by the finitely many entries of the generators.

Each local field is either archimedean (and isomorphic to $\mathbb{R}$ and $\mathbb{C}$ with their standard topologies) or non-archimedean (and totally disconnected).

Culler proved in \cite[Theorem 3.1]{Culler} that discrete subgroups of $\pslc$ with no 2-torsion can be lifted to $\slc$. Here we prove the following non-archimedean analogon of this result. Although every non-archimedean local field of characteristic zero (i.e. a finite extension of $\mathbb{Q}_p$) has the same first order theory as $\mathbb{C}$, the differing topologies mean that, even in characteristic $0$, we require a different approach here than that used by Culler.

\begin{thmx}\label{main}
Let $K$ be a non-archimedean local field with $\chk \neq 2$. A discrete subgroup $G$ of $\pslk$ can be lifted to a subgroup of $\slk$ if and only if $G$ has no $2$-torsion.
\end{thmx}

We additionally obtain an analogon of Fricke's result that many Fuchsian groups, including all non-cocompact ones, are free products of a free group and a free product of cyclic groups amalgamated over an infinite cyclic subgroup \cite[p. 98]{Magnus}. For Kleinian groups the situation is far more complicated.

\begin{thmx}\label{char0}
Let $K$ be a non-archimedean local field of characteristic zero. If $G$ is a discrete subgroup of $\pslk$ with no $2$-torsion, then it is a free product whose possible factors are cyclic groups and direct products of a finite cyclic and an infinite cyclic group.
\end{thmx}

Button constructed explicit examples of indiscrete subgroups of $\mathrm{PSL}(2,\mathbb{C})$, $\mathrm{PSU}(2)$ and $\mathrm{PSL}(2,\mathbb{R})$ without 2-torsion which cannot be lifted up to $\mathrm{SL}(2,\mathbb{C})$ \cite{Button}. Here we prove a non-archimedean analogue of \cite[Theorems 2 \& 5]{Button}. 

\begin{thmx}\label{no-lift}
Let $K$ be a non-archimedean local field with $\chk \neq 2$. There exists a subgroup $G$ of $\pslk$ which does not lift to $\slk$, and whose non-trivial torsion elements are all unipotent. In particular, $G$ is torsion free if $\chk=0$ and, if $\chk=p$, then all possible non-trivial torsion elements have order $p$.
\end{thmx}

\begin{remark}
Throughout the paper we refer to elements of $\pslk$ as unipotent if one of their lifts is; in particular, if the lifts are $\{U,-U\}$ for some unipotent element $U$ of $\slk$.
\end{remark}

We conclude the paper by briefly discussing non-archimedean analogons of \cite[Propositions 1 and 2]{Button-add}. Given an abstract group $\Gamma$, we say that a representation $\rho:\Gamma\to \pslk$ {\it lifts} to a representation $\bar{\rho}:\Gamma\to \slk$ if $\rho =\pi\circ \bar{\rho}$. If $\rho$ is faithful, then this is equivalent to being able to lift $\rho(\Gamma)$. Note that if $\chk=2$, then every representation of a group into $\pslk$ lifts to $\slk$ since $\pi$ is the identity map.

\begin{prpx}\label{lift-rep}
Let $K$ be a non-archimedean local field. If $\Gamma$ is a group such that $H^2(\Gamma,\mathbb{Z}/2\mathbb{Z})=0$, then every representation $\rho$ of $\Gamma$ into $\pslk$ can be lifted to $\slk$.
\end{prpx}

As is the case over $\mathbb{C}$, this condition is certainly not necessary:

\begin{prpx}\label{lift-rep-non-zero}
Let $K$ be a non-archimedean local field. If $n \geq 7$, or $n=6$ and $\chk \neq 3$, then every representation of the alternating group $A_n$ into $\pslk$ lifts to $\slk$, but $H^2(A_n,\mathbb{Z}/2\mathbb{Z})\neq 0$.
\end{prpx}

\subsection*{Acknowledgements}
We thank Sam Hughes for cohomological advice.
This work has received funding from the European Research Council (ERC) under the European Union's Horizon 2020 research and innovation programme (Grant agreement No. 850930). It has also received funding from the New Zealand Marsden Fund (grant numbers 3723563 and 3725880) and the Rutherford Foundation (grant number 3726956). The first author would like to thank the University of Auckland for its hospitality during a visit in September 2023. 
\section{Background}

Let $K$ be a non-archimedean local field with finite residue field of order $q=p^r$. We denote the valuation ring of $K$ by $\mathcal{O}_K$, the corresponding valuation map by $v$, and we let $u$ be a uniformiser of $K$.
We use $T_K$ to denote the corresponding Bruhat--Tits tree, upon which $\pslk$ (and hence $\slk$) acts by isometries, without inversion. Each element of $\pslk$ (and their representatives in $\slk$) can therefore be classified as either elliptic or hyperbolic, depending on its translation length on $T_K$. We will use $\fix(g)$ to denote the fixed point set of an elliptic element $g$, and $\ax(h)$ to denote the translation axis of a hyperbolic element $h$. We respectively refer to \cite{Cassels} and \cite{Serre} for more information on non-archimedean local fields and the Bruhat--Tits tree.

\subsection{Action of $\pslk$ on the Bruhat--Tits tree}\label{BT-tree}

We begin by recalling properties of stabilisers and fixed point sets for the action of $\pslk$ on the Bruhat--Tits tree $T_K$.

\begin{lem}{\cite[Lemma 4.4.1]{K}}\label{disc-stab}
Let $G$ be a subgroup of $\pslk$.
\begin{enumerate}[label={$(\arabic*)$}]
\item If $G$ is discrete, then ${\rm Stab}_G(y)$ is finite for every vertex $y$ of $T_K$.
\item If ${\rm Stab}_G(y)$ is finite for some vertex $y$ of $T_K$, then $G$ is discrete.
\end{enumerate}
\end{lem}

Using the classification of finite subgroups of $\pslk$ given in \cite[Theorem 1]{VM}, we obtain:

\begin{prop}\label{finite}
If $G$ is a finite subgroup of $\pslk$ with no $2$-torsion, then either:
\begin{itemize}
\item $G$ is cyclic of odd order, or
\item $G$ is a semi-direct product of an elementary abelian $p$-group with a cyclic group of order dividing $q-1$, and $\chk=p\neq 2$;
\end{itemize}
\end{prop}

Following \cite[Section 2.4]{Vanderput}, we say  that a group satisfying the second condition is of {\it Borel type}, as it can be represented in $\slk$ by a group of upper triangular matrices. Since the boundary of $T_K$ can be identified with the projective line $\mathbb{P}^1(K)$ \cite[p. 72]{Serre}, it follows that a group of Borel type necessarily fixes an end of $T_K$.

As a direct consequence of \Cref{disc-stab} and \Cref{finite}, we obtain the following.

\begin{coro}\label{all-cyclic}
Let $G$ be a discrete subgroup of $\pslk$ with no $2$-torsion. If $\chk \in \{0,2\}$, then ${\rm Stab}_G(y)$ is a finite cyclic group for every vertex $y$ of $T_K$.
\end{coro}

As in \cite{BC}, we define a {\it band} of radius $r$ and {\it nerve} $S$ to be the subtree of $T_K$ consisting of all vertices at distance at most $r$ from some geodesic $S$. Given an infinite ray $R$ emanating from a vertex $x$ of $T_K$, we define a {\it horoball} to be the subtree of $T_K$ consisting of all vertices at distance at most $d(v,x)$ from each vertex $v$ of $R$.

\begin{prop}\label{fix-classes}
If $g\in \pslk$ is non-trivial with finite order $n$, then either
\begin{itemize}
\item $\fix(g)$ is a band with nerve $S$, which is either a vertex, an edge or a bi-infinite ray. Its radius, and the type of $S$, depend only on $n$ and $K$; or
\item $\fix(g)$ is a horoball and $\chk=p=n$. 
\end{itemize}
\end{prop}

\begin{proof}
Let $\overline{g}$ be a representative of $g$ in $\slk$, and note that $\fix(g)=\fix(\overline{g})$. If $\overline{g}$ is unipotent, then we may conjugate it into the form
$$
\pm \left[ \begin{array}{cc}
1 & x \\ 
0 & 1 \end{array} \right],
$$
and $\fix(g)$ is the horoball defined by the ray $\mathcal{O}_Ke_1+u^{-v(x)+n}\mathcal{O}_Ke_2$ for $n \ge 0$ \cite[Proposition 20]{BC}. Since $g$ has finite order, $K$ must have characteristic $p$, whence $g$ has order $p$. 
We may assume $\overline{g}$ is diagonalisable over either $K$ or a quadratic extension $K'$ of $K$. In the first case, $\fix(g)$ is a band with nerve the unique bi-infinite ray between the ends of $T_K$ determined by the eigenvectors of $\overline{g}$. In the latter case, we consider the bi-infinite ray $S'$ between the ends of the Bruhat--Tits tree $T_{K'}$ corresponding to $K'$ determined by the eigenvectors of $\overline{g}$. The intersection $S' \cap T_K$ is either a vertex or the midpoint of an edge, and that vertex/edge is the nerve of $\fix(g)$; see \cite[Proposition 18]{BC} for further details. 
\end{proof}

\begin{coro}\label{preserve-nerve}
If $g,h\in \pslk$ are finite order elements which commute, then either:
\begin{itemize}
\item $\fix(g)$ and $\fix(h)$ are both bands with the same nerve; or
\item $\fix(g)$ and $\fix(h)$ are both horoballs with the same fixed end.
\end{itemize}
In particular, either $\fix(g) \subseteq \fix(h)$ or $\fix(h) \subseteq \fix(g)$.
\end{coro}
\begin{proof}
If $g$ is represented by a unipotent element in $\slk$, then the representatives of $g$ and $h$ in $\slk$ can be simultaneously conjugated into upper triangular unipotent matrices with top right entries $x_g$ and $x_h$ respectively. The proof of \Cref{fix-classes} shows that $\fix(g)$ and $\fix(h)$ are both horoballs, and that $\fix(g) \subseteq \fix(h)$ if $v(x_g) \le v(x_h)$ and $\fix(h) \subseteq \fix(g)$ otherwise.

On the other hand, if $g$ is diagonalisable over $K$ or a quadratic extension $K'$ of $K$, then $g$ and $h$ are simultaneously diagonalisable and have the same eigenvectors. Thus the proof of \Cref{fix-classes} shows that $\fix(g)$ and $\fix(h)$ have the same nerve, and the result follows.
\end{proof}

Given an elliptic element $g \in \pslk$, let $F(g) = \bigcup_{i\in \mathbb{Z}} \set{\fix(g^i) \given g^i \neq 1}$.

\begin{lem}\label{key-collapse-gen}
Suppose $\chk = 0$. Let $G$ be a discrete subgroup of $\pslk$ with no $2$-torsion. If $g,h \in G$ are elliptic, then either $F(g)\cap F(h)=\varnothing$, $\fix(g)\subseteq \fix(h)$ or $\fix(h)\subseteq \fix(g)$.
\end{lem}
\begin{proof}
First note that $g,h$ have finite order by \Cref{disc-stab}. Suppose that $F(g)\cap F(h) \neq \varnothing$, so that there is some vertex $v$ of $T_K$ which is fixed by some non-trivial powers $g^i$ and $h^j$. By \Cref{all-cyclic}, $g^i$ and $h^j$ lie in a common cyclic subgroup of $\pslk$ and hence commute. Thus \Cref{preserve-nerve} shows that $\fix(g^i)$ and $\fix(h^j)$ are both bands with the same nerve. Another application of \Cref{preserve-nerve} shows that $\fix(g)$ and $\fix(g^i)$ are also both bands with the same nerve, and similarly for $\fix(h)$ and $\fix(h^j)$. Hence $\fix(g)$ and $\fix(h)$ are bands with the same nerve, so either $\fix(g)\subseteq \fix(h)$ or $\fix(h)\subseteq \fix(g)$.
\end{proof}

\subsection{Bass--Serre theory}\label{Bass-Serre}

Let $Y$ be a directed graph. For each edge $e$ of $Y$, define $\overline{e}$ to be the reverse of $e$ and $i(e)$ to be the initial vertex of $e$. We recall that a {\it graph of groups} $\G = (G, Y)$ consists of the following:

\begin{itemize}
  \item For each vertex $v \in V(Y)$, a group $G_v$. If $T$ is a spanning tree of $Y$, then $G$ is isomorphic to the fundamental group $\pi_1(\G, T)$. Also note that $G_v \cong G_w$ for all $w \in V(X)$ representing $v \in V(Y)$; see \cite[I.5]{Serre} for details. This is the situation we consider below.
  \item For each edge $e \in E(Y)$, a group $G_e$ and an injective map $\sigma_e \colon G_e \to G_{i(e)}$, with the requirement that $G_e = G_{\overline{e}}$.
\end{itemize}

Let $T$ be a spanning tree of $Y$. Recall that the {\it fundamental group} $\pi_1(\G, T)$ is the group generated by $\set{G_v \given v \in V(Y)}$ and a set of symbols $\set{t_e \given e \in E(Y)}$ satisfying the following relations:

\begin{itemize}
  \item $t_{\overline{e}} = t_e^{-1}$.
  \item If $e \in E(T)$, then $t_e = 1$.
  \item $t_e \sigma_e(g) t_e^{-1} = \sigma_{\overline{e}}(g)$ for all $g \in G_e$.
\end{itemize}

If $G$ is a group which acts without inversions on a tree $X$, and $Y$ is the quotient graph $G \backslash X$, then we naturally obtain a quotient graph of groups $\G = (G, Y)$ by considering the vertex and edge stabilisers of $X$ in $G$. To make this precise, we choose a subtree $S$ of $X$ that is a lift of a maximal tree $T$ of $Y$, and take as orbit representatives the vertices and edges of $S$, adding an edge for every orbit $e$ not yet accounted for so that $\overline{e}$ is represented by the reverse of the edge representing $e$, and so that one endpoint lies in $S$. As vertex and edge groups take the stabilisers of the orbit representatives; if an edge is in $T$ then it is naturally a subgroup of its adjacent vertices and we take these embeddings. Otherwise, (fixing the orientation of $e$ so that $i(\overline{e})$ lies in $S$) take $\sigma_{\overline{e}}$ to be the natural embedding between the stabilisers, and choose a conjugating element $g_e$ bringing $i(e)$ into $S$. Then let the embedding $\sigma_e$ be the inclusion of stabilisers followed by conjugation by $g_e$.

One can define the Bass--Serre tree for a graph of groups, which admits an action of the fundamental group $\pi_1(\G,T)$. The utility of these constructions follows from the structure theorem of Bass--Serre theory:
\begin{theorem}[{\cite[Section 5.4, Theorem 13]{Serre}}]\label{B-S}
	Up to isomorphism of the structures concerned, the construction of the quotient graph of groups from an action on a tree is mutually inverse to the construction of the Bass--Serre tree and fundamental group.
\end{theorem}

\subsection*{Collapsing Maps}
Given an action of $G$ on a tree $T$, we can obtain a new action of $G$ on another tree $T'$ by collapsing each member of some $G$-orbit(s) of edges to a point. This process usually changes the elliptic subgroups: the effect on the graph of groups is to replace any collapsed subgraphs with their (Bass--Serre theoretic) fundamental group, which up to isomorphism is the setwise stabiliser of some largest subtree containing only edges which are collapsed. If one has sufficient control over these subtrees and their setwise stabilisers, then this can clarify the algebraic structure of $G$. There is a $G$-equivariant collapse map $f: T \to T'$ given by sending each edge and vertex in a collapsed subtree to the vertex corresponding to it. (See for instance \cite[Section 3]{GLdeformation} for a discussion of these and related ideas involved in comparing actions of the same group on different trees.)

\section{Proofs of Theorems \ref{main} and \ref{char0}}

Before proving Theorems \ref{main} and \ref{char0}, we first show that under certain conditions vertex and edge stabilisers in $\pslk$ have unique lifts to $\slk$.

\begin{lem}\label{stab-lift}
Let $G$ be a discrete subgroup of $\pslk$ with no $2$-torsion. For each vertex $v$ and edge $e$ of $T_K$, there is a unique lift of $G_{v}$ and $G_e$ to $\slk$. 
\end{lem}

\begin{proof}
Let $G_x$ be the stabiliser in $G$ of either a vertex or an edge of $T_K$. By \Cref{disc-stab}, $G_x$ is a finite group.
  First we show that at most one lift of $G_x$ to $\slk$ exists. Let $g \in G_x$ have order $k$.
	Since $G$ and hence $G_x$ has no 2-torsion, $k$ is odd and exactly one of the pre-images $\pi^{-1}(g) = \{g', -g'\}$ in $\slk$ has order $k$.
	Thus each element of $G_x$ has a unique lift, so a lift of $G_x$, if it exists, must be unique.
	
	Next we show that such a lift exists. By \Cref{finite}, $G_x$ is either cyclic or of Borel type.
  If $G_x$ is cyclic, then the lift of $G_x$ is determined by the lift of a generator of $G_x$. Hence we may suppose that $G_x$ is of Borel type and $\chk=p$. By conjugating if necessary, we may assume that $G_x$ fixes $\infty$ on the projective line $\mathbb{P}^1(K)$. 	
  
  Let $\phi \colon G_x \to \slk$ be the map such that $\phi(g) \in \pi^{-1}(g)$ has the same order as $g$. Note that $\phi(G_x)$ is upper triangular since $G_x$ fixes $\infty$. There exists an odd integer $m$ coprime to $p$ such that the diagonal entries of each element of $\phi(G_x)$ are $m$-th roots of unity. Hence $\phi(G_x)$ is closed under multiplication and thus a group. Since $-I \notin \phi(G_x)$, we conclude that $\phi$ is a lift of $G_x$.
\end{proof}

\begin{proof}[Proof of \Cref{main}]
Let $G$ be a discrete subgroup of $\pslk$. It suffices to show that, if $G$ has no 2-torsion, then it can be lifted to a subgroup of $\slk$. Let $Y$ be the quotient graph $G \backslash T_{K}$ with spanning tree $T$, and consider the corresponding graph of groups $\G = (G, Y)$.
 
 Observe that the set $\set{t_e \given e \in E(Y) \setminus E(T)}$ (up to restricting to an orientation of each edge) freely generates a free subgroup of $\pi_1(\G, T)$.
   Let $W$ be the disjoint union of $G_v$ and $t_e$ for each vertex $v$ and edge $e$ of $Y$. In other words, $W$ is the set of symbols in the standard presentation for $\pi_1(\G, T)$.
  Define a map $\phi \colon W \to \slk$ on generators by letting $\phi$ restrict to the map given by \Cref{stab-lift} on each $G_v$. For each $t_e$, we choose $\phi(t_e)$ to be an arbitrary lift of $t_e$ such that $\phi(t_{\overline{e}}) = \phi(t_e)^{-1}$. If $t_e$ is trivial in $G$, then necessarily $\phi(t_e) = 1$. If $t_e$ is nontrivial in $G$ (and hence of infinite order), then $\phi(t_e)$ is arbitrarily chosen from the two possible lifts; the particular choice made does not affect the argument.

  Let $H$ be the group generated by the image of $\phi$. We see that the projection of $\slk$ onto $\pslk$ induces a surjection $\psi \colon H \twoheadrightarrow G$.
  
  We prove that $\phi$ sends every relation in $G$ to a relation in $H$.
  Firstly, $\phi(t_{\overline{e}}) = \phi(t_e)^{-1}$ by construction.
  Secondly, if $e \in E(T)$, then $t_e = 1$ so $\phi(t_e) = 1$.
  Finally, let $e \in E(Y)$.
  
We identify each $g \in G_e$ as an element of $G$ (and hence of $\pslk$) via $g=\sigma_{\overline{e}}(g)$. 
Define maps $f_1, f_2: \sigma_e(G_e) \to \slk$ by
\[
f_1(\sigma_e(g)) = \phi(t_e)\phi(\sigma_e(g))\phi(t_e^{-1}), \qquad f_2(\sigma_e(g)) = \phi(\sigma_{\overline{e}}(g)).
\]
We need to check that (for every edge $e$) $f_1=f_2$. Notice that the relation $t_e \sigma_e(g) t_e^{-1} = \sigma_{\overline{e}}(g)$ implies that $f_2(\sigma_e(g)) = \phi(t_e \sigma_e(g) t_e^{-1})$ for $g \in G_e$. Let $\theta$ be the inner automorphism of $G$ given by $\theta(h) = t_e^{-1} h t_e$ for each $h \in G$. Note that $\theta$ is surjective and it sends $\sigma_{\overline{e}}(G_e)$ isomorphically to $\sigma_e(G_e)$. Now consider the compositions $\pi \circ f_1 \circ \theta$ and $\pi \circ f_2 \circ \theta$, where we restrict to $G_e=\sigma_{\overline{e}}(G_e)$ 
so that these maps are well-defined. Observe that, since $\pi$ is a homomorphism and $\pi \circ \phi$ is the identity map on $W$, both compositions are the identity map; in particular $f_1 \circ \theta$ and $f_2 \circ \theta$ are both lifts of $G_e$. Hence by \Cref{stab-lift} these two lifts are equal. Since $\theta$ is surjective, $f_1=f_2$ as we needed to show.
\end{proof}

\begin{remark}
In general it is not true that a lift on vertex stabilisers can be extended to a lift of the whole group; the uniqueness appears to be necessary unless one has very fine control. Consider the quotient map from the Baumslag--Solitar group $\mathrm{BS}(4,6) =\langle a,t : ta^4t^{-1} = a^6 \rangle $ to $\mathrm{BS}(2,3) =\langle b,s : sb^2s^{-1} =b^3 \rangle $, sending $a \mapsto b$ and $t \mapsto s$. Note that $\mathrm{BS}(2,3)$ acts on a tree with one orbit each of edges and vertices (as an HNN extension of $\mathbb{Z}$ with respect to its index 2 and 3 subgroups), and both the vertex group and edge generators lift to the generators of $\mathrm{BS}(4,6)$, but this cannot extend to a lift of the whole group since the groups are not isomorphic. The failure of uniqueness is witnessed by elements such as $ta^2t^{-1}a^{-2}$, which is a loxodromic element of $\mathrm{BS}(4,6)$ descending to $b$ in $\mathrm{BS}(2,3)$.

\end{remark}

\begin{proof}[Proof of \Cref{char0}]
Let $G$ be a discrete subgroup of $\pslk$ with no 2-torsion, where $\chk=0$. Let $X$ be the tree obtained from $T_K$ by collapsing every edge with non-trivial stabiliser.
Note that the collapsing map $f \colon T_K \to X$ is equivariant,
and so $X$ inherits an action of $G$ from $T_K$.

 We observe that, since all fixed point sets were collapsed, every edge stabiliser of $X$ in $G$ is trivial.
	
Now let $v$ be a vertex of $X$. We first show that $f^{-1}(v)=\fix(g)$ for some $g\in G$. 
Indeed, if $x,y \in f^{-1}(v)$, then there is a path joining them where every edge has non-trivial stabiliser. In particular, there is some sequence of vertices $x=x_0, x_1, \dots ,x_n=y$ of $T_K$ where each consecutive pair $\{x_i,x_{i+1}\}$ is contained in some $\fix(g_i)$, and we may assume that $g_i \neq g_{i+1}$ (if not, $x_{i+1}$ may be omitted). This implies that $x_{i+1} \in \fix(g_i) \cap \fix(g_{i+1})$, whence \Cref{key-collapse-gen} implies that either $\fix(g_i) \subseteq \fix(g_{i+1})$, or $\fix(g_{i+1}) \subseteq \fix(g_{i})$. In either case, both $x_i$ and $x_{i+1}$ are contained in the larger fixed point set, and we may remove $x_{i+1}$ from the sequence. Continuing in this fashion, the sequence is eventually reduced to the two vertices $x$ and $y$, which must be contained in $\fix(g)$ for some $g \in G$. Note that $g$ necessarily has finite order by \Cref{disc-stab}.

Now let $H$ be the set-wise stabiliser in $G$ of $f^{-1}(v)=\fix(g) \subseteq T_K$. By \Cref{fix-classes}, $\fix(g)$ is a band with nerve $S$ which is either a vertex, an edge or a bi-infinite ray. Note that $H$ must preserve $S$. If $S$ is either a vertex or an edge, then $H$ is a finite cyclic group by \Cref{all-cyclic}. On the other hand, if $S$ is a bi-infinite ray, then each element of $H$ either fixes or swaps the ends of $S$. Since $G$ contains no elements of order 2, we deduce 
that every element of $H$ fixes both ends of $S$. Hence we obtain a set of diagonalisable matrices with the same eigenvectors, so they commute and are simultaneously diagonalisable. If $H$ contains no hyperbolic elements then $H$ fixes a point of $S$ and is hence cyclic by \Cref{all-cyclic}. If $H$ does contain hyperbolic elements stabilising $S$ then it contains one of smallest translation length, hence the hyperbolic elements of $H$ form an infinite cyclic group. Since $H$ is abelian it is hence either a finite cyclic group or a direct product of a finite cyclic and an infinite cyclic group. 
The conclusion follows by applying \Cref{B-S}: as remarked above the proof edge stabilisers are trivial, hence the edge relations in the fundamental group are empty, and so $G$ is a free product of its vertex stabilisers and a free group.
\end{proof}

\begin{remark}
The difficulty with duplicating this argument in positive characteristic lies in the subgroups of Borel type, and understanding how their fixed point sets can intersect fixed point sets of other subgroups: it is no longer necessarily true that the subtrees $f^{-1}(v)$ coincide with fixed point sets. 
\end{remark}

\section{Proof of Theorem \ref{no-lift}}

We recall the classification of local fields \cite[Proposition II.5.2]{N}.
Let $K$ be a non-archimedean local field. If $\chk=0$, then $K$ is a finite extension of $\qp$ and $\mathbb{Q}$ is the corresponding {\it prime field}. If $\chk>0$, then $K$ is a field of formal Laurent series $\fq((t))$ and $\fp$ is the corresponding prime field. 
The following two basic lemmas are well-known. 
\begin{lem}\label{infinite-trans-degree}
Every non-archimedean local field has infinite transcendence degree over its corresponding prime field.
\end{lem}
\begin{proof}
For the characteristic zero case, see \cite[Chapter 5, Lemma 2.3]{Cassels}. In positive characteristic, note that the formal power series over a field $k$ are in bijection with the maps $\mathbb{N}\to k$. Hence its fraction field, the Laurent series over $k$, is uncountable. 
\end{proof}

\begin{lem}\label{squares-open}
Let $K$ be a non-archimedean local field with $\chk \neq 2$. The set of non-zero squares in $K$ forms an open subset of $K$.
\end{lem}
\begin{proof}
Since multiplication is continuous, it suffices to show that there is an open neighbourhood of $1\in K$ consisting entirely of squares. We closely follow the arguments of \cite[Chapter 4, Lemmas 3.2 and 3.3]{Cassels}, but in a more general setting. Recall that $q$ is the size of the residue field of $K$, and let $|-|$ denote the non-archimedean absolute value corresponding to $K$ given by $|x|=q^{-v(x)}$.

If $q$ is odd, then suppose $a \in K$ is such that $|a-1|<1$. In particular, $v(a-1)>0$ and the ultrametric inequality implies that $a \in \oo$. Consider the polynomial $f(x)=x^2-a\in \oo[x]$, and note that $|f'(1)|=|2|=1$. Thus $|f(1)|<|f'(1)|^2$, and it follows from Hensel's Lemma \cite[Chapter 4, Lemma 3.1]{Cassels} that $f(x)$ has a root in $\oo \subseteq K$, whence $a$ is a square. 

If $q$ is even, then suppose $a \in K$ is such that $|a-1|<q^{-3}$. Again, the ultrametric inequality implies that $a \in \oo$, and we consider the polynomial $f(x)=x^2-a \in \oo[x]$. Observe that $|f'(1)|=|2|=q^{-1}$. Hence $|f(1)|<|f'(1)|^2$, so $a$ is a square by Hensel's Lemma \cite[Chapter 4, Lemma 3.1]{Cassels}.
\end{proof}

Recall that we may assume $\chk \neq 2.$ Let $i=\sqrt{-1}$ and consider the field extension $K(i)$ of $K$. Note that $K(i)=K$ whenever $q \equiv 1 \mod 4$ by Hensel's Lemma \cite[Chapter 4, Lemma 3.1]{Cassels}.

Let $S$ denote the set of all $(A,B,C,D) \in \slki^4$ such that $A$ is diagonal and 
\begin{align}
	ABA^{-1}B^{-1}CDC^{-1}D^{-1} &= -I. \label{minus-surfae-eq}
\end{align}
This is an analogue to the set $X_0$ defined in \cite{Button}. We equip $S$ with the subspace topology inherited from $K(i)^{16}$.
 
Each quadruple $(A,B,C,D) \in S$ generates a subgroup of $\slki$ whose image in $\pslki$ does not lift. To see this, note that in $\pslki$ this relation descends to the relation $ABA^{-1}B^{-1}CDC^{-1}D^{-1} = I$. But all choices of lifts among $\{\pm A, \pm B, \pm C, \pm D\}$ evaluate to $-I$, so the subgroup of $\pslki$ does not lift. 

Given $(A,B,C,D) \in \slki^4$ and a word $w$ in the free group $F_4$ of rank four, we use $w(A,B,C,D)$ to denote the word $w$ evaluated on $(A,B,C,D)$, and $\pi w(A,B,C,D)$ to denote its projection to $\pslki$.

The proof of the lemma below follows the same general method of proof as \cite[Lemma 1]{Button}, however, here we prove a more general statement and provide explicit examples of matrices.

\begin{lem}\label{no-order-2}
For each word $w \in F_4$, there exists $(A,B,C,D) \in S$ such that $\pi w(A,B,C,D)$ is not of order $2$. 
\end{lem}
\begin{proof}
Suppose for a contradiction that there exists $w \in F_4$ such that $\pi w(A,B,C,D)$ has order 2 for every $(A,B,C,D) \in S$. 
Let $a,b,c,d\in \{0,1\}$ be the respective parities of the exponent sums of $A,B,C,D$ in $w(A,B,C,D)$.  We consider some particular choices of $(A,B,C,D) \in S$ to obtain the desired contradiction.
Throughout, we will use $u$ to denote a uniformiser of $K(i)$.

First consider the following set of matrices $(A,B,C,D) \in S$:

$$A = \begin{bmatrix} i & 0 \\ 0 & -i \end{bmatrix}, B =  \begin{bmatrix} 0 & 1 \\ -1 & 0 \end{bmatrix}, C= \frac{1}{2u}\begin{bmatrix} u^2+1 & u^2-1 \\ u^2-1 & u^2+1 \end{bmatrix}, D = \begin{bmatrix} 1 & 0 \\ 0 & 1 \end{bmatrix}.$$

Observe that $A$ and $B$ generate $Q_8$, the quaternion group with order $8$, and that $ACA^{-1}=BCB^{-1}=C^{-1}$. 
Moreover, \cite[Proposition II.3.15]{MS} implies that $C$ is hyperbolic (with respect to $T_{K(i)}$) and hence of infinite order.
The respective images $\alpha, \beta, \gamma, \delta$ of $A,B,C,D$ in $\pslki$ generate a subgroup with the following presentation:
\[\langle \alpha, \beta, \gamma : \alpha^2=\beta^2=[\alpha,\beta]=1, \alpha\gamma\alpha =  \beta\gamma\beta = \gamma^{-1} \rangle.\]

There are four possible normal forms in this group for $\pi w(A,B,C,D)$, that is $\gamma^k, \alpha\gamma^k, \beta\gamma^k, \alpha\beta\gamma^k$ for $k \in \mathbb{Z}$, and passing to such a normal form preserves the parities $a,b,c,d$. Of these, $\alpha\gamma^k$ and $\beta\gamma^k$ are involutions for any value of $k$, $\gamma^k$ is either infinite order or trivial, and $\alpha\beta\gamma^k$ is infinite order unless $k=0$, when it is an involution. This rules out the six parities where either $a=b=0$ or $a=b=c=1$. Repeating this argument with the roles of $C$ and $D$ interchanged additionally rules out $(a,b,c,d)=(1,1,0,1)$.

Now consider the following two sets of matrices $(A,B,C,D) \in S$: 
$$A = \begin{bmatrix} 1 & 0 \\ 0 &1 \end{bmatrix}, B =  \frac{1}{2u}\begin{bmatrix} u^2+1 & u^2-1 \\ u^2-1 & u^2+1 \end{bmatrix}, C= \begin{bmatrix} i & 0 \\ 0 & -i \end{bmatrix}, D = \begin{bmatrix} 0 & 1 \\ -1 & 0 \end{bmatrix}; $$
$$A = \begin{bmatrix} u & 0 \\ 0 &1/u \end{bmatrix}, B = \begin{bmatrix} 1 & 0 \\ 0 &1 \end{bmatrix}, C= \begin{bmatrix} 0 & -i \\ -i & 0 \end{bmatrix}, D = \begin{bmatrix} 0 & 1 \\ -1 & 0 \end{bmatrix}. $$

In the first set, $C$ and $D$ generate $Q_8$, and $B$ is a hyperbolic element such that $CBC^{-1}=DBD^{-1}=B^{-1}$. The respective images $\alpha, \beta, \gamma, \delta$ of $A,B,C,D$ in $\pslk$ generate a subgroup with the following presentation:
\[\langle \beta, \gamma, \delta : \gamma^2=\delta^2=[\gamma,\delta]=1, \gamma\beta\gamma =  \delta\beta\delta = \beta^{-1} \rangle.\]
A similar argument to the above rules out the parities where $c=d=0$ or $b=c=d=1$. In the second set of matrices, the roles of $A$ and $B$ are interchanged, while retaining the property that $A$ is diagonal. Observe that $A$ is hyperbolic by \cite[Proposition II.3.15]{MS}, and that $CAC^{-1}=DAD^{-1}=A^{-1}$. Hence similar reasoning shows that this rules out the parities $(a,b,c,d)=(1,0,1,1)$. 

This leaves four available choices of parities: $(1,0,1,0)$, $(1,0,0,1)$, $(0,1,1,0)$ and $(0,1,0,1)$. To this end, we consider the following two sets of matrices $(A,B,C,D) \in S$:
$$A = \begin{bmatrix} i & 0 \\ 0 & -i \end{bmatrix}, B =  \begin{bmatrix} 0 & 1 \\ -1 & 0 \end{bmatrix}, C= \begin{bmatrix} u & 0 \\ 0 &1/u \end{bmatrix}, D = \begin{bmatrix} 1 & 0 \\ 0 & 1 \end{bmatrix};$$

$$A = \begin{bmatrix} i & 0 \\ 0 & -i \end{bmatrix}, B =  \begin{bmatrix} 0 & 1 \\ -1 & 0 \end{bmatrix}, C= \frac{-1}{2u}\begin{bmatrix} u^2+1 & i(1-u^2) \\ i(u^2-1) & u^2+1 \end{bmatrix}, D = \begin{bmatrix} 1 & 0 \\ 0 & 1 \end{bmatrix}.$$

In both these cases, $A$ and $B$ again generate $Q_8$ and $C$ is hyperbolic. Observe that $ACA^{-1}=BC^{-1}B^{-1}=C$ in the first case, and $AC^{-1}A^{-1}=BCB^{-1}=C$ in the second case. If the parities were $(1,0,1,0)$, the only possible normal form for $\pi w(A,B,C,D)$ in the first case is $AC^k$ (with $k$ odd). However, this element has infinite order. Similarly, we can use the second case to rule out the parity $(0,1,1,0)$. The remaining two parities can be dealt with similarly, by interchanging the roles of $C$ and $D$.
\end{proof}

While we needed to work over the field extension $K(i)$ to obtain explicit matrices in the previous lemma, we now return to considering matrices in $\slk$.

\begin{lem}\label{dense-set}
There is a dense subset $Y$ of $S$, and an element $y^* \in Y \cap \slk^4$, such that if $\pi w (y^*)$ has order $2$ for some word $w \in F_4$ then $\pi w(y)$ has order $2$ for every $y \in Y$.
\end{lem}
\begin{proof}
Following \cite{Button}, we consider matrices in $\slk$ of the form 
\[A=\begin{bmatrix}  \lambda & 0\\ 0 & \frac{1}{\lambda} \end{bmatrix}\:, B =\begin{bmatrix} a & b \\ c & d \end{bmatrix}\:, C = \begin{bmatrix} \alpha_1 & \beta_1 \\ \gamma_1 & \delta_1 \end{bmatrix} \:, D = \begin{bmatrix} \alpha_2 & \beta_2 \\ \gamma_2 & \delta_2 \end{bmatrix}\]
Now let
$$CDC^{-1}D^{-1}=\begin{bmatrix} \alpha & \beta \\ \gamma & \delta \end{bmatrix}, $$
so that $\alpha, \beta, \gamma, \delta$ are polynomials in the variables $\alpha_i, \beta_i, \gamma_i, \delta_i$ for $i \in \{1,2\}$. 

Let $S'$ be the set of matrices $(A,B,C,D)$
subject to the conditions
\begin{align}\label{conditions}
\alpha+\delta \neq -2,  \hspace{0.5cm} b,\beta \neq 0, \hspace{0.5cm} \alpha, \delta \neq -1.
\end{align}

To ensure that $S'$ is a subset of $S$ (that is, \Cref{minus-surfae-eq} is satisfied),
we can compute that:
\begin{equation}\label{entries}
  \begin{gathered}[c]
    \lambda^2=-\frac{1+\delta}{1+\alpha}, \\
    c=-\frac{(\alpha+1)(\delta+1)}{b(\alpha+\delta+2)},
  \end{gathered} \hspace{1.5cm}
  \begin{gathered}[c]
    a=-\frac{\beta(1+\alpha)}{b(\alpha+\delta+2)}, \\
    d=\frac{b(\alpha\delta-1)}{\beta(1+\alpha)}.
  \end{gathered}
\end{equation}

Let $Y$ be the subset of $S'$ where $\beta_1,\beta_2 \neq 0$. Note that $Y$ is dense in $S$. We may additionally write $\gamma_i=\frac{\alpha_i\delta_i-1}{\beta_i}$ for each $i \in \{1,2\}$, showing that all entries of $A, B, C, D$ can be expressed as rational functions of the eight variables $\alpha_1, \alpha_2, \beta_1, \beta_2, \delta_1, \delta_2, b, \lambda$ with coefficients in the prime field $F$ of $K$. 

In order to find an element of $Y \cap \slk^4$, for $\lambda \in K \backslash \{0,\pm 1\}$, we choose
\[C = \begin{bmatrix} \frac{\lambda^2+3}{1-\lambda^2} & 1\\-1 & 0\end{bmatrix}, D = \begin{bmatrix} 1 & 2 \\ 0 & 1 \end{bmatrix}.\]
Observe that $\alpha=1+\frac{2\lambda^2+6}{1-\lambda^2}$ and $\delta=5-\frac{2\lambda^2+6}{1-\lambda^2}$ so that $\lambda^2=-\frac{1+\delta}{1+\alpha}$. After excluding the solutions (if any) to $(\lambda^2+3)^2=8$, to ensure $\beta \neq 0$, all above conditions are satisfied in all characteristics except $\chk=2$, which is excluded from the start. Now choose some $b \in K^*$, so that using \Cref{entries} we obtain an element $y=(A,B,C,D) \in Y \cap \slk^4$.

Next we claim that we can find a nearby solution $y^* \in Y \cap \slk^4$ such that the seven values $\alpha_1, \alpha_2, \beta_1, \beta_2, \delta_1, \delta_2, b$ are algebraically independent. Suppose, for a contradiction, that $I$ is an algebraically independent set of maximal size for which a solution exists and let $k\notin I$ be one of the seven. Since non-zero squares in local fields form an open set by \Cref{squares-open}, and open balls are uncountable, and there are only countably many algebraic numbers in $F(I)$, we can slightly modify $k$ to find a nearby solution still respecting \ref{entries} and \ref{conditions} with $I\cup \{k\}$ algebraically independent, a contradiction.

Now suppose that $\pi w(y^*)$ has order 2 for some $w \in F_4$. Hence $\tr(w(y^*))=0$, giving a polynomial $P \in F[\alpha_1, \alpha_2, \beta_1, \beta_2, \delta_1, \delta_2, b, \lambda]$ such that $P(y^*)=0$. As in \cite{Button}, we may temporarily suppress the first seven variables and write this polynomial in the form $P=P_1(\lambda^2)+\lambda P_2(\lambda^2)$. It follows that 
\[[(P_1(\lambda^2) + \lambda P_2(\lambda^2))(P_1(\lambda^2) - \lambda P_2(\lambda^2))] (y^*)= (P_1^2(\lambda^2)-\lambda^2P_2^2(\lambda^2))(y^*)=0.\] 
Since $\lambda^2=-\frac{1+\delta}{1+\alpha}$ and the first seven variables are algebraically independent, we deduce that either $P_1(\lambda^2) + \lambda P_2(\lambda^2)$ or $P_1(\lambda^2) - \lambda P_2(\lambda^2)$ is identically zero. Since these polynomials can be obtained from each other by replacing $\lambda$ with $-\lambda$, we conclude that $\tr(w(y))=0$ for every $y \in Y$, giving the result.
\end{proof}

\begin{proof}[Proof of \Cref{no-lift}]
Choose $y^* \in Y \cap \slk^4$ as in \Cref{dense-set}, and let $G$ be the subgroup of $\pslk$ generated by the image of $y^*$. Note that $G$ does not lift to $\slk$.

Suppose that there is a word $w \in F_4$ for which $\pi w(y^*)$ has non-trivial finite order. Let $f$ be the minimal annihilating polynomial of $\tr(w(y^*))$ over the prime field of $K$. Hence we obtain the equation $f(\tr(w(y^*)))f(\tr(w(-y^*)))=0$, giving a polynomial $P \in F[\alpha_1, \alpha_2, \beta_1, \beta_2, \delta_1, \delta_2, b, \lambda]$ such that $P(y^*)=0$. As in the previous proof, we use the algebraic independence of the first seven variables and the relation $\lambda^2=-\frac{1+\delta}{1+\alpha}$ to deduce that $P$ is identically zero. Thus either $\tr(w(y))$ or $-\tr(w(y))$ is a root of $f$ for every $y \in Y$.

Now consider $A=B=I$ and $$C= \begin{bmatrix} 0 & -i \\ -i & 0 \end{bmatrix}, D = \begin{bmatrix} 0 & 1 \\ -1 & 0 \end{bmatrix}, $$
so that $(A,B,C,D) \in Y$ and every element of $\langle A,B,C,D \rangle$ has trace 0 or $\pm 2$. Since $f$ is irreducible we deduce that $f(x) \in \{x, x- 2, x+2\}$. This implies that $w(y^*)$ either has order 2 or it is unipotent. If $w(y^*)$ has order 2, then \Cref{dense-set} implies that $\tr(w(y))=0$ for each $y \in Y$. Since $Y$ is dense in $S$, continuity of the trace function implies that $\tr(w(s))=0$ for each $s \in S$. This contradicts \Cref{no-order-2}. If $w(y^*)$ is unipotent, we obtain a contradiction if $\chk=0$ and observe that $w(y^*)$ must have order $p$ if $\chk=p$.
\end{proof}

\section{Proof of Propositions \ref{lift-rep} and \ref{lift-rep-non-zero}} 

The following proof is identical to the argument for $\pslc$ in \cite{Button-add} and many more places besides; we include it for completeness.

\begin{proof}[Proof of \Cref{lift-rep}]
	Given a representation $\rho: \Gamma \to \pslk$, we can pull it back along the quotient $\pi: \slk \to \pslk$ to get a representation $\rho': \Gamma' \to \slk$, with $\Gamma'$ a central extension of $\Gamma$ by $\zz/2\zz$. The situation is as in this commutative diagram; both vertical rows are the middle three terms of a central extension.
	\begin{figure}[h!]
	\begin{tikzcd}
		\zz/2\zz \arrow{d} & \zz/2\zz \arrow{d} \\
		\Gamma' \arrow{d} \arrow{r}{\rho'} & \slk \arrow{d}{\pi} \\
		\Gamma \arrow{r}{\rho} \arrow[bend left,dashed]{u} & \pslk \\
	\end{tikzcd}
	\end{figure}
Central extensions by $\zz/2\zz$ are classified by the cohomology group $H^2(\Gamma;\zz/2\zz)$; since this is trivial, there is only the trivial extension and $\Gamma'\cong\zz/2\zz \times \Gamma$.  In particular, following any section $\Gamma \to \Gamma'$ by $\rho'$ gives a lift of $\rho$.
\end{proof}

Our final proof differs from \cite[Proposition 2]{Button-add} in that we also consider the alternating groups $A_6$ and $A_7$, and require an additional argument in the positive characteristic setting.

\begin{proof}[Proof of \Cref{lift-rep-non-zero}]
We follow the same argument as \cite[Proposition 2]{Button-add}. First note that $A_n$ is simple for $n \ge 5$ and $H^2(A_n; \mathbb{Z}/2\mathbb{Z})=\mathbb{Z}/2\mathbb{Z} \neq 0$ for $n \ge 4$, which follows from their Schur multipliers \cite[page 27]{Wilson} and a short calculation with the Universal Coefficient Theorem \cite[Theorem 3.2]{Hatcher}. Furthermore, for each $n \ge 6$, the group $A_n$ does not occur as a subgroup of $\pslk$, except for $A_6 \cong {\rm PSL_2(9)}$ when $\chk=3$. Indeed, the only feasible possibilities for an isomorphism between $A_n$ and the subgroups of $\pslk$ listed in \cite[Theorem 1]{VM} are with a group of Borel type, or a group of the form ${\rm PGL}(2,q')$ or ${\rm PSL}(2,q')$ for some divisor $q'$ of $q$, all of which occur only in positive characteristic. The first two possibilities would contradict that $A_n$ is simple and, other than the case $A_6 \cong {\rm PSL_2}(9)$, the last possibility is ruled out by \cite[page 3]{Wilson}. 
Hence, when $n\ge 7$ or $n=6$ and $\chk \neq 3$, the only representation of $A_n$ into $\pslk$ is the identity representation. This can certainly be lifted, which completes the proof.
\end{proof}

\end{document}